\documentclass{amsart}
\usepackage [latin1] {inputenc}

\usepackage {mathrsfs, amsmath,amsthm,  amssymb}
\usepackage{dsfont}
\usepackage [intlimits] {empheq}
\usepackage{paralist}

\newtheorem{thm}{Theorem}[section]
\newtheorem{prop}[thm]{Proposition}
\newtheorem{defn}[thm]{Definition}
\newtheorem{lem}[thm]{Lemma}

\newtheorem{rem}[thm]{Remark}

\renewcommand{\leq}{\leqslant}
\renewcommand{\geq}{\geqslant}

\title{The Multiplicities of Root Number Functions}
\author{Stefan-Christoph Virchow}

\begin{document}

\begin{abstract} 
We consider the $q$th root number function for the symmetric group. Our aim is to develop an asymptotic formula for the multiplicities of the $q$th root number function as $q$ tends to $\infty$. We use character theory, number theory and combinatorics. 
\end{abstract}

\maketitle

\section{Introduction}
Let $q$ be a positive integer. We define the $q$th root number function $r_q\colon S_n\to\mathbb{N}_0$ via
\begin{align*}
r_q(\pi):=\#\{\sigma\in S_n:~\sigma^q=\pi\}.
\end{align*}
Obviously, $r_q$ is a class function of the symmetric group $S_n$. For each irreducible character $\chi$ of $S_n$ let
\begin{equation*}
m_{\chi}^{(q)}:=\langle r_q,\chi\rangle
\end{equation*}
be the multiplicity of $\chi$ in $r_q$.\par

Scharf \cite{Sc1} proved that the $q$th root number functions $r_q$ are proper characters, that is, the multiplicities $m_{\chi}^{(q)}$ are non-negative integers. A good account of results on root number functions can be found in \cite[Chapter 6.2 and 6.3]{Ke1}.
\par
We now pay attention to the multiplicities $m_{\chi}^{(q)}$. Müller and Schlage-Puchta established estimates for the multiplicities $m_{\chi_{\lambda}}^{(q)}$ (cf. \cite[Proposition 2 and 3]{Mu1}). In addition, they showed the following: Let $\Delta\in\mathbb{N}$ be fixed and let $q\geq 2$ be an integer. Given a partition $\mu\vdash\Delta$, there exists some constant $C_{\mu}^q$, depending only on $\mu$ and $q$, such that for $n$ sufficiently large and for a partition $\lambda=(\lambda_1,\ldots,\lambda_l)\vdash n$ with $\lambda\backslash \lambda_1:=(\lambda_2,\ldots,\lambda_l)=\mu$, we have $m_{\chi_{\lambda}}^{(q)}=C_{\mu}^q$. In particular, we get
\begin{align*}
C_{(1)}^q&=\sigma_0(q)-1\\
C_{(2)}^q&=\frac{1}{2}(\sigma_1(q)+\sigma_0(q)^2-3\sigma_0(q)+\sigma_0'(q))\\
C_{(1,1)}^q&=\frac{1}{2}(\sigma_1(q)+\sigma_0(q)^2-3\sigma_0(q)-\sigma_0'(q))+1,
\end{align*}
where $\sigma_0'(q)$ is the number of odd divisors of $q$.\par
Our aim is to generalize this result. We establish an asymptotic formula for the multiplicities $m_{\chi_{\lambda}}^{(q)}$ as $q$ tends to $\infty$. More precisely, we claim the following:

\begin{thm}
\label{3-1-1}
Let $q\in\mathbb{N}$ be sufficiently large and let $\Delta\in\mathbb{N}$ with $\Delta\leq \frac{\log q}{\log 2}$. In addition, let $n\geq \Delta q$ be an integer. Then, for partitions $\lambda\vdash n$ and $\mu\vdash \Delta$ with $\lambda\backslash\lambda_1=\mu$, we have
\begin{align*}
m_{\chi_{\lambda}}^{(q)}=\begin{cases}
\sigma_0(q)+\mathcal{O}(1) &\text{if }\Delta=1\\[0.5ex]
\frac{1}{2}\sigma_1(q)+\mathcal{O}\left((\sigma_0(q))^2\right) &\text{if }\Delta=2\\[0.5ex]
\frac{\chi_{\mu}(1)}{6}\sigma_2(q)+\mathcal{O}(\sigma_0(q)\sigma_1(q)) &\text{if }\Delta=3\\[0.5ex]
\frac{\chi_{\mu}(1)}{\Delta!}\sigma_{\Delta-1}(q)+\mathcal{O}\left(\frac{\chi_{\mu}(1)}{\Delta!}q^{\Delta-2}\left(\Delta\sigma_0(q)+2^{\Delta}\right)\right) &\text{if }\Delta\geq 4.
\end{cases}
\end{align*}
\par\bigskip\noindent
The $\mathcal{O}$-constant is universal.
\end{thm}
\begin{rem}
\label{3-1-2}
The error term in our asymptotic formula is essentially optimal.
\end{rem}
\noindent
\par\bigskip
The proof of our Theorem proceeds as follows: At  first, we realize that for $\pi\in S_n$ the value $\chi_{\lambda}(\pi)$ is a polynomial in the number $c_i(\pi)$ of $i$-cycles of the permutation $\pi$ for $i=1,\ldots,n$. We use this result to establish a formula with main and error term for $m_{\chi_{\lambda}}^{(q)}$, where the random variables $c_i$ appear again. Secondly, we summarize identities and estimates for Stirling numbers of the first and second kind and then we review  bounds for the divisor function. Thirdly, we examine the distribution of cycle in $S_n$ and compute the mean of $(c_{k_1})^{m_1}\cdot\ldots\cdot(c_{k_j})^{m_j}$. The formula for the mean includes Stirling numbers of the second kind. Finally, we calculate the main and the error term of $m_{\chi_{\lambda}}^{(q)}$ obtained in the first step using the outcomes of step two and three.
\par\bigskip\noindent
\textbf{Some notation.} A \emph{partition} of $n$ is a sequence $\lambda=(\lambda_1,\lambda_2,\ldots,\lambda_l)$ of positive integers such that $\lambda_1\geq\lambda_2\geq\ldots\geq\lambda_l$ and $\lambda_1+\lambda_2+\ldots+\lambda_l=n$. We write $\lambda\vdash n$ to indicate that $\lambda$ is a partition of $n$. By $\lambda\backslash \lambda_1$ we mean the partition $\lambda\backslash \lambda_1=(\lambda_2,\lambda_3,\ldots,\lambda_l)$. The \textit{weight} $|\lambda|$ of $\lambda$ is $|\lambda|=\sum_{j=1}^l\lambda_j$. For partitions $\lambda,\mu$ we write $\mu\subset\lambda$, if $\mu_j\leq\lambda_j$ for all $j$.\\
We denote by $\chi_{\lambda}$ the irreducible character of the symmetric group $S_n$ corresponding to the partition $\lambda$ of $n$.\\
For a permutation $\pi\in S_n$ and $1\leq i\leq n$ let $c_i(\pi)$ be the number of $i$-cycles of $\pi$.\\
Furthermore, for $\alpha\in\mathbb{R}$ let 
\begin{equation*}
\sigma_{\alpha}(q):=\sum_{d|q}d^{\alpha}
\end{equation*}
be the \emph{divisor function}. As usual, we denote by $\zeta(s)$ the Riemann zeta function and we write $(n,k)$ for the greatest common divisor of $n$ and $k$.\\
We denote by ${n\brack k}$ and ${n\brace k}$ the \textit{Stirling numbers of the first and second kind}, respectively. Finally, $(x)_n$ denotes the falling factorial.

\section{Proof of the Theorem}
In this section, we use character theory to derive a formula with main and error term for the multiplicities $m_{\chi_{\lambda}}^{(q)}$. We apply this result to prove our Theorem.\\
Müller and Schlage-Puchta \cite[Lemma 7]{Mu1} established the following.
\begin{lem}
\label{3-2-1}
Let $\lambda\vdash n$ be a partition, $\mu=\lambda\backslash\lambda_1$, and let $\pi\in S_n$ be a permutation. Then
\begin{equation*}
\chi_{\lambda}(\pi)=\sum_{\substack{\tilde{\mu}\subseteq\mu\\\tilde{\mu}_1\leq 1}}(-1)^{|\tilde{\mu}|}\sum_{\boldsymbol{c}\subseteq S_{|\mu|-|\tilde{\mu}|}}\chi_{\mu,\tilde{\mu}}(\boldsymbol{c})\prod_{i\leq|\mu|}\binom{c_i(\pi)}{c_i},
\end{equation*}
where $\boldsymbol{c}$ runs over all conjugacy classes of $S_{|\mu|-|\tilde{\mu}|}$, $\chi_{\mu,\tilde{\mu}}(\boldsymbol{c})$ denotes the number of ways to obtain $\tilde{\mu}$ from $\mu$ by removing rim hooks according to the cycle structure of $\boldsymbol{c}$, counted with the sign prescribed by the Murnaghan-Nakayama rule\footnote{Cf. for instance \cite[Theorem 4.10.2]{Sa1}}, and $c_i$ is the number of $i$-cycles of an element of $\boldsymbol{c}$.
\end{lem}
\noindent
This result shows that $\chi_{\lambda}(\pi)$ is a polynomial in $c_i(\pi)$ for $i=1,\ldots,|\mu|$ with leading term $\chi_{\mu}(1)(|\mu|!)^{-1}c_1(\pi)^{|\mu|}$. We now observe:

\begin{lem}
\label{3-2-2}
Let $\lambda\vdash n$ and $\mu\vdash\Delta$ be partitions with $\mu=\lambda\backslash\lambda_1$, and let $\pi\in S_n$ be a permutation. Then we have
\begin{equation*}
\chi_{\lambda}(\pi)=\chi_{\mu}(1)\binom{c_1(\pi)}{\Delta}+\mathcal{O}\left(\chi_{\mu}(1)\sum_{j=1}^{\Delta}\binom{c_1(\pi)+\ldots+c_{j+1}(\pi)}{\Delta-j}\right).
\end{equation*}
\end{lem}

\begin{proof}
Applying Lemma \ref{3-2-1}, we obtain

\begin{equation*}
\chi_{\lambda}(\pi)=\chi_{\mu}(1)\binom{c_1(\pi)}{\Delta}+\sum_{\substack{\tilde{\mu}\subseteq\mu\\\tilde{\mu_1}\leq 1}}(-1)^{|\tilde{\mu}|}\sum_{\boldsymbol{c}\subseteq S_{\Delta-|\tilde{\mu}|}}\chi_{\mu,\tilde{\mu}}(\boldsymbol{c})\prod_{i\leq\Delta}\binom{c_i(\pi)}{c_i},
\end{equation*}
where $\boldsymbol{c}$ runs over all conjugacy classes of $S_{\Delta-|\tilde{\mu}|}$ except the trivial class of $S_{\Delta}$. Therefore, we realize the expected main term. We shall show that the second term in the above formula can be absorbed into the error term.\\
At first, we observe that $|\chi_{\mu,\tilde{\mu}}(\boldsymbol{c})|\leq \chi_{\mu}(1)$ for a conjugacy class $\boldsymbol{c}$ of $S_{\Delta-|\tilde{\mu}|}$. Secondly, let $\boldsymbol{c}$ be a conjugacy class of $S_k$ with $1\leq k\leq \Delta$ and let $c_i$ be the number of $i$-cycles of an element of $\boldsymbol{c}$. Suppose that $c_1+c_2+\ldots+c_{\Delta}=\Delta-j$ for some positive integer $j$. Then we have $c_i=0$ for all $i\geq j+2$.\\
Therefore, it follows that the absolute value of the considered second term in the preceding formula is bounded above by
\begin{equation*}
\chi_{\mu}(1)\sum_{j=1}^{\Delta}\sum_{\substack{(c_1,\ldots,c_{j+1})\in\mathbb{N}_0^{j+1}\\1c_1+\ldots+(j+1)c_{j+1}\leq\Delta\\c_1+\ldots+c_{j+1}=\Delta-j}}\prod_{i\leq j+1}\binom{c_i(\pi)}{c_i}\leq\chi_{\mu}(1)\sum_{j=1}^{\Delta}\binom{c_1(\pi)+\ldots+c_{j+1}(\pi)}{\Delta-j}.
\end{equation*}
This yields our assertion.
\end{proof}
\noindent
Due to 
\begin{equation*}
m_{\chi_{\lambda}}^{(q)}=\frac{1}{n!}\sum_{\pi\in S_n}\chi_{\lambda}(\pi^q),
\end{equation*}
we obtain as immediate consequence of the preceding Lemma the following result.

\begin{prop}
\label{3-2-3}
Let $\lambda\vdash n$ and $\mu\vdash\Delta$ be partitions with $\mu=\lambda\backslash\lambda_1$, and let $q\in\mathbb{N}$. Then we get
\begin{equation*}
m_{\chi_{\lambda}}^{(q)}=\frac{\chi_{\mu}(1)}{n!}\sum_{\pi\in S_n}\binom{c_1(\pi^q)}{\Delta}+\mathcal{O}\left(\frac{\chi_{\mu}(1)}{n!}\sum_{j=1}^{\Delta}\sum_{\pi\in S_n}\binom{c_1(\pi^q)+\ldots+c_{j+1}(\pi^q)}{\Delta-j}\right).
\end{equation*}
\end{prop}

\noindent
Now, we give the
\begin{proof}[Proof of Theorem \ref{3-1-1}]
We stated a formula with main and error term for $m_{\chi_{\lambda}}^{(q)}$ in Proposition \ref{3-2-3}. In section 5, we will evaluate the main term: see Proposition \ref{3-5-3}. In section 6, we will estimate the error term: cf. Lemma \ref{3-6-2}.  Therefore, the proof of our Theorem is completed. Moreover, the error term in our Theorem is essentially optimal due to the Remark \ref{3-5-4}.
\end{proof}
\noindent
In the next two sections, we shall establish some auxiliary results.

\section{Combinatorics and number theory}

In this section we review some results about Stirling numbers of the first and second kind as well as basic facts about the divisor function.

\begin{defn}
Let $n$ and $k$ be positive integers. The \emph{Stirling numbers of the second kind} $n\brace k$ count the number of ways to partition a set of $n$ labeled objects into $k$ nonempty unlabeled subsets.
\end{defn}

\begin{lem}
\label{3-3-1}
Let $n$ and $k$ be positive integers.
\begin{compactenum}[1)]
\item 
We have ${n\brace 2}=2^{n-1}-1$ and ${n\brace n-1}=\binom{n}{2}$.
\item 
In addition, we state the recurrence
\begin{equation*}
k!{n\brace k}=k^n-\sum_{j=1}^{k-1}\frac{k!}{(k-j)!}{n\brace j}.
\end{equation*}
\end{compactenum}
\end{lem}

\begin{proof}
1) follows from the definition. For the proof of 2) see \cite[Theorem 7.2.6]{Mo1}.
\end{proof}
\noindent
This recurrence yields an upper bound for $n\brace k$. Next, we would like to represent the ordinary powers $x^n$ by falling factorials $(x)_k:=x(x-1)\cdot\ldots\cdot (x-k+1)$.

\begin{lem}
\label{3-3-2}
Let $n$ be a positive integer. Then the identity
\begin{equation*}
x^n=\sum_{k=1}^n {n\brace k}(x)_k
\end{equation*}
holds.
\end{lem}

\begin{proof}
See for instance \cite[Formula (6.10)]{Gr1}.
\end{proof}
\noindent
Now, we pay attention to the Stirling numbers of the first kind.
\begin{defn}
Let $n$ and $k$ be positive integers. The \emph{Stirling numbers of the first kind} $n\brack k$ count the number of ways to arrange $n$ objects into $k$ cycles. So $n\brack k$ equals the number of permutations of $n$ elements with exactly $k$ disjoint cycles.
\end{defn}

\begin{lem}
\label{3-3-3}
Let $n$ and $k$ be positive integers.
\begin{compactenum}[1)]
\item
We have the recurrence
\begin{equation*}
{n\brack k}=(n-1){n-1\brack k}+{n-1\brack k-1}.
\end{equation*}
\item
We obtain the estimate
\begin{equation*}
{n\brack k}\leq\frac{(n-1)!}{k!}\binom{n}{k-1}.
\end{equation*}
\end{compactenum}
\end{lem}

\begin{proof}
1) Cf. \cite[Formula (6.8)]{Gr1}.\\
2) By induction over $n$: Obviously, the estimate is true for $n\leq k$. Applying the recurrence 1) yields for $n\geq k$
\begin{equation*}
{n+1\brack k}\leq\frac{n!}{k!}\left(\binom{n}{k-1}+\frac{k}{n}\binom{n}{k-2}\right)\leq\frac{n!}{k!}\binom{n+1}{k-1}.
\end{equation*}
\end{proof}
\noindent
Stirling numbers of the first kind are (up to sign) the coefficients of ordinary powers that yield the falling factorial $(x)_n$. More precisely, we get
\begin{lem}
\label{3-3-4}
Let $n$ be a positive integer. Then the identity
\begin{equation*}
(x)_n=\sum_{k=1}^n (-1)^{n-k}{n\brack k}x^k
\end{equation*}
holds.
\end{lem}
\begin{proof}
See for instance \cite[Formula (6.13)]{Gr1}.
\end{proof}

\noindent
Finally, we state results about the divisor function.

\begin{lem}
\label{3-3-5}
\begin{compactenum}[1)]
\item
Let $\epsilon>0$. Then we have $~~~\sigma_0(q)\leq (2+\epsilon)^{\frac{\log q}{\log \log q}}~~~$ for all $q\geq q_0(\epsilon)$.
\item
$\sigma_1(q)\ll q\log\log q~~~$ for all $q\geq q_0$.
\item
Let $k\geq 2$. Then $~~~q^k\leq\sigma_k(q)\leq \zeta(2)q^k.$
\end{compactenum}
\end{lem}

\begin{proof}
For 1) and 2) see \cite[Theorem 317 and Theorem 323]{Ha1}.\\
3) The lower bound is obvious. The upper bound follows from the fact
\begin{equation*}
\sigma_k(q)=q^k\sum_{d|q}\frac{1}{d^k}.
\end{equation*}
\end{proof}

\section{Statistics of the symmetric group}
Müller and Schlage-Puchta \cite[Lemma 13]{Mu1} showed that, for $\pi\in S_n$ chosen at random, the distribution of $c_k(\pi)$ converges to a Poisson distribution with mean $\frac{1}{k}$ as $n\to\infty$. In addition, they proved that the mean of $\left(c_k(\cdot)\right)^m$ converges to $\sum_{s=1}^m{m\brace s}k^{-s}$ as $n\to\infty$. We generalize this result and make it more explicit.

\begin{prop}
\label{3-4-1}
Let $k_1,\ldots,k_l$ be distinct positive integers and let $m_j\in\mathbb{N}$ for $j=1,\ldots,l$. Then
\begin{equation*}
\frac{1}{n!}\sum_{\pi\in S_n}\prod_{j=1}^l\left(c_{k_j}(\pi)\right)^{m_j}\leq \prod_{j=1}^l\left(\sum_{s=1}^{m_j}{m_j\brace s}k_j^{-s}\right).
\end{equation*}
If $~~\sum_{j=1}^l k_j m_j\leq n~~$ is fulfilled, then we have an equality.
\end{prop}

\begin{proof}
1) Let $k_1,\ldots,k_l$ be distinct positive integers and let $s_j\in\mathbb{N}$ for $j=1,\ldots,l$. We observe
\begin{equation*}
\sum_{\pi\in S_n}\prod_{j=1}^l\binom{c_{k_j}(\pi)}{s_j}=
\begin{cases}
0&\textit{if }\sum_{j=1}^l k_j s_j>n\\
n!\left(\prod_{j=1}^l k_j^{s_j}s_j!\right)^{-1}&\textit{if }\sum_{j=1}^l k_j s_j\leq n.
\end{cases}
\end{equation*}
You can see this equality as follows.\\
\textit{Case 1:} Let $~\sum_{j=1}^l k_j s_j>n.~~~$ Then there exists no $\pi\in S_n$ such that $c_{k_j}(\pi)\geq s_j$ for all $j=1,\ldots,l$. Therefore, the considered sum is equal to $0$.\\
\textit{Case 2:} Let $~\sum_{j=1}^l k_j s_j\leq n.~~~$ Then the left hand side of the equation is equal to the number of tuples 
$(\tau_1,\ldots,\tau_{l+1})$, which satisfy the following condition: There exists distinct, disjoint cycles $\sigma_{i j}$ and a non-negative integer $s_{l+1}$ such that $\tau_j=\sigma_{1j}\cdot\ldots\cdot\sigma_{s_jj}$ for all $j=1,\ldots,l+1$, the cycles $\sigma_{ij}$ have length $k_j$ for all $j=1,\ldots,l$ and $\prod_{j=1}^{l+1}\prod_{i=1}^{s_j}\sigma_{ij}$ is the cycle decomposition for a permutation from $S_n$. Finally, the number of these tuples is equal to
\begin{align*}
\frac{n!}{k_1^{s_1}s_1!(n-k_1s_1)!}&\frac{(n-k_1s_1)!}{k_2^{s_2}s_2!(n-k_1s_1-k_2s_2)!}\cdot\ldots\\
&\qquad\qquad\ldots\cdot\frac{\left(n-\sum_{j=1}^{l-1}k_js_j\right)!}{k_l^{s_l}s_l!\left(n-\sum_{j=1}^{l}k_js_j\right)!}\biggl(n-\sum_{j=1}^{l}k_js_j\biggr)!
\end{align*}
Canceling yields our assertion.\\
2) It follows from 1) that
\begin{equation*}
\frac{1}{n!}\sum_{\pi\in S_n}\prod_{j=1}^l\left(c_{k_j}(\pi)\right)_{s_j}=
\begin{cases}
0&\textit{if }\sum_{j=1}^l k_j s_j>n\\
\prod_{j=1}^l k_j^{-s_j} &\textit{if }\sum_{j=1}^l k_j s_j\leq n.
\end{cases}
\end{equation*}

\noindent
3) Eventually, we compute the desired mean of a product of random variables $c_k(\cdot)$. Applying Lemma \ref{3-3-2} and 2) yields
\begin{align*}
\operatorname{E}\Biggl(\prod_{j=1}^l\left(c_{k_j}\right)^{m_j}\Biggr)&=\sum_{\substack{(s_1,\ldots,s_l)\\1\leq s_j\leq m_j}}\operatorname{E}\Biggl(\prod_{j=1}^l(c_{k_j})_{s_j}\Biggr)\prod_{i=1}^l{m_i\brace s_i}\\
&\leq \sum_{\substack{(s_1,\ldots,s_l)\\1\leq s_j\leq m_j}}\prod_{j=1}^l \biggl(k_j^{-s_j}{m_j\brace s_j}\biggr)\\
&=\prod_{j=1}^l\left(\sum_{s=1}^{m_j}{m_j\brace s}k_j^{-s}\right).
\end{align*}
Obviously, we have an equality if $~~\sum_{j=1}^l k_j m_j\leq n~~$ is fulfilled.
\end{proof}
\noindent
Furthermore, Müller and Schlage-Puchta \cite[Formula (33)]{Mu1} established the following useful identity.
\begin{lem}
\label{3-4-2}
Let $d$ and $q$ be positive integers and let $\pi\in S_n$. Then
\begin{equation*}
c_d(\pi^q)=\sum_{\substack{k\\k/(k,q)=d}}(k,q)c_k(\pi).
\end{equation*}
\end{lem}

\section{The main term}
We carry out the first step of the plan formulated in the proof at the end of section 2: We compute the main term in Proposition \ref{3-2-3}. At first, we draw our attention to the mean of a power of $c_1(\pi^q)$.
\begin{lem}
\label{3-5-1}
Let $q\in\mathbb{N}$ be sufficiently large and let $\delta\in\mathbb{N}$ with $\delta\leq \frac{\log q}{\log 2}$. In addition, let $n\geq \delta q$ be an integer. Then we obtain
\begin{align*}
\frac{1}{n!}\sum_{\pi\in S_n}\Bigl(c_1(\pi^q)\Bigr)^{\delta}=\begin{cases}
\sigma_0(q) &\text{if }\delta=1\\[0.5ex]
\sigma_1(q)+(\sigma_0(q))^2 &\text{if }\delta=2\\[0.5ex]
\sigma_2(q)+\mathcal{O}(\sigma_0(q)\sigma_1(q)) &\text{if }\delta=3\\[0.5ex]
\sigma_{\delta-1}(q)+\mathcal{O}\Bigl(q^{\delta-2}\left(\delta\sigma_0(q)+2^{\delta}\right)\Bigr) &\text{if }\delta\geq 4.
\end{cases}
\end{align*}
\par\bigskip\noindent
The $\mathcal{O}$-constant is universal.
\end{lem}

\begin{proof}
1) At first, we consider the case $\delta\in\{1,2\}$. We sketch the argument for $\delta=2$ (the case $\delta=1$ is similar). Using Lemma \ref{3-4-2} we get

\begin{equation*}
\frac{1}{n!}\sum_{\pi\in S_n}\Bigl(c_1(\pi^q)\Bigr)^2=\sum_{k|q}k^2\frac{1}{n!}\sum_{\pi\in S_n}\left(c_k(\pi)\right)^2+\sum_{\substack{(k_1,k_2)\\k_i|q\\k_1\neq k_2}}k_1k_2\frac{1}{n!}\sum_{\pi\in S_n}c_{k_1}(\pi)c_{k_2}(\pi).
\end{equation*}
Since $n\geq 2q$, it follows with Proposition \ref{3-4-1} that the considered mean is equal to $\sigma_1(q)+\left(\sigma_0(q)\right)^2$. This shows our claim for $\delta=2$.\par\medskip\noindent
2) We generalize this method for an arbitrary $\delta$. Let $n\geq \delta q$. Applying Lemma \ref{3-4-2} and Proposition \ref{3-4-1} yields
\begin{equation*}
\frac{1}{n!}\sum_{\pi\in S_n}\Bigl(c_1(\pi^q)\Bigr)^{\delta}=\sum_{l=1}^{\delta}\sum_{\{M_1,\ldots,M_l\}}\sum_{\substack{(k_1,\ldots,k_l)\\k_i|q\\ k_i\neq k_j~(i\neq j)}}\prod_{j=1}^l\sum_{s=1}^{|M_j|}{|M_j|\brace s}k_j^{|M_j|-s},
\end{equation*}
where the second sum on the right is over all set partitions of $\{1,\ldots\delta\}$ in exactly $l$ sets $M_1,\ldots,M_l$.\par\medskip\noindent
3) We direct our attention to
\begin{equation*}
T_m:=\sum_{k|q}\sum_{s=1}^m{m\brace s}k^{m-s}=\sum_{s=1}^m{m\brace s}\sigma_{m-s}(q).
\end{equation*}
\textit{Let $q$ be sufficiently large and $m\leq \frac{\log q}{\log 2}$. Then}
\begin{equation*}
T_m\leq
\begin{cases}
\sigma_0(q)&\textit{if } m=1\\
\sigma_0(q)+\sigma_1(q)&\textit{if } m=2\\
3q^{m-1}&\textit{if } m\geq 3.
\end{cases}
\end{equation*}
\textit{In particular, we have} $T_2\leq q\sigma_0(q)$.\\
You can see this estimate as follows: The cases $m=1$ and $m=2$ are obvious. Let $m\geq 3$. It results from Lemma \ref{3-3-1} and Lemma \ref{3-3-5} for a constant $C>0$
\begin{equation*}
T_m\leq Cm^2q\log\log q+\zeta(2)q^{m-1}\sum_{s=1}^{m-2}{m\brace s}q^{-s+1}.
\end{equation*}
Applying the estimate ${m\brace s}\leq s^m (s!)^{-1}$ (see Lemma \ref{3-3-1}) we find that
\begin{equation*}
\sum_{s=1}^{m-2}{m\brace s}q^{-s+1}\leq e-1
\end{equation*}
\noindent
This proves our assertion.
\par\medskip\noindent
4) \textit{Let $q$ be sufficiently large and $\delta\leq\frac{\log q}{\log 2}$. Then, for positive integers $m_i$ such that $m_1+\ldots+m_l=\delta$, we have}
\begin{equation*}
\prod_{j=1}^lT_{m_j}\leq q^{\delta-l}\Bigl(\max\{3,\sigma_0(q)\}\Bigr)^l.
\end{equation*}
This results immediately from 3).
\par\medskip\noindent
5) Using the outcomes of step 2) and 4), we obtain
\begin{align*}
\frac{1}{n!}\sum_{\pi\in S_n}\Bigl(c_1(\pi^q)\Bigr)^{\delta}&=\sum_{k|q}\sum_{s=1}^{\delta}{\delta\brace s}k^{\delta-s}+\mathcal{O}\left(\sum_{l=2}^{\delta}\sum_{\{M_1,\ldots,M_l\}}\prod_{j=1}^l T_{|M_j|}\right)\\[1.3ex]
&=\sigma_{\delta-1}(q)+\mathcal{O}\Bigl(F_1+F_2+F_3\Bigr),
\end{align*}
where
\begin{align*}
F_1&:=\sum_{s=2}^{\delta}{\delta\brace s}\sigma_{\delta-s}(q),\\
F_2&:=\sum_{\{M_1,M_2\}}T_{|M_1|}T_{|M_2|},\\
F_3&:=\sum_{l=3}^{\delta}{\delta\brace l}q^{\delta-l}\Bigl(\operatorname{max}\{3,\sigma_0(q)\}\Bigr)^l.
\end{align*}
The sum in $F_2$ is over all set partitions of $\{1,\ldots,\delta\}$ in exactly two sets $M_1, M_2$.\\
Therefore, we realize the expected main term. We shall show that the error term is sufficiently small. 
\par\medskip\noindent
6) For the rest of the proof let $3\leq\delta\leq\frac{\log q}{\log 2}$. Lemma \ref{3-3-5} and Lemma \ref{3-3-1} yield
\begin{equation*}
F_1\ll
\begin{cases}
\sigma_1(q) &\textit{if }\delta=3\\
2^{\delta}q^{\delta-2} &\textit{if }\delta\geq 4
\end{cases}
\end{equation*}
and
\begin{equation*}
F_3\ll2^{\delta}q^{\delta-2},
\end{equation*}
which is sufficiently small.
\par\medskip\noindent
7) Finally, we examine $F_2$. The term $F_2$ determines the order of the error term. More precisely, we get
\begin{align*}
F_2&\leq \binom{\delta}{1}T_1T_{\delta-1}+\binom{\delta}{2}T_2T_{\delta-2}+\sum_{\substack{\{M_1,M_2\}\\|M_i|\geq3}}T_{|M_1|}T_{|M_2|}\\
&\ll
\begin{cases}
\sigma_0(q)\sigma_1(q)&\textit{if }\delta=3\\
\sigma_0(q)q^2&\textit{if }\delta=4\\
q^{\delta-2}\left(\delta\sigma_0(q)+2^{\delta}\right)&\textit{if }\delta\geq 5.
\end{cases}
\end{align*}
In the last estimate, we used the outcome of step 3). In addition, we applied the inequality $(\sigma_1(q))^2\leq\frac{9}{8}\sigma_0(q)q^2$ for $\delta=4$ and the fact that $\delta^2\sigma_1(q)q^{-1}\leq\delta^2\left(\log\sigma_0(q)+1\right)\ll\max\{2^{\delta},\delta\sigma_0(q)\}$ for the case $\delta\geq 5$. Therefore, the proof is completed.
\end{proof}
\noindent
Our next aim is to show that the result stated in Lemma \ref{3-5-1} is essentially optimal.

\begin{lem}
\label{3-5-2}
Let $q$ be prime and let $\delta\leq \frac{\log q}{\log 2}$ be a positive integer. In addition, let $n\geq \delta q$ be an integer. Then we get
\begin{align*}
\frac{1}{n!}\sum_{\pi\in S_n}\Bigl(c_1(\pi^q)\Bigr)^{\delta}=\begin{cases}
\sigma_0(q) &\textit{if }\delta=1\\[0.5ex]
\sigma_1(q)+(\sigma_0(q))^2 &\textit{if }\delta=2\\[0.5ex]
\sigma_{\delta-1}(q)+q^{\delta-2}\left(\delta+2^{\delta-1}-1\right)+\mathcal{O}\left(3^{\delta} q^{\delta-3}\right) &\textit{if }\delta\geq 3.
\end{cases}
\end{align*}
\par\bigskip\noindent
The $\mathcal{O}$-constant is universal.
\end{lem}

\begin{proof}
1) For $\delta\in\{1,2\}$ see Lemma \ref{3-5-1}. So let $\delta\geq3$. Since $q$ is prime, Lemma \ref{3-4-2} and Proposition \ref{3-4-1} yield
\begin{align*}
\frac{1}{n!}\sum_{\pi\in S_n}\Bigl(c_1(\pi^q)\Bigr)^{\delta}&=\frac{1}{n!}\sum_{\pi\in S_n}\sum_{k=0}^{\delta}\binom{\delta}{k}(c_1(\pi))^{\delta-k}(qc_q(\pi))^k\\
&=\sigma_{\delta-1}(q)+q^{\delta-2}(\delta+2^{\delta-1}-1)+F_1+F_2+F_3,
\end{align*}
where
\begin{align*}
F_1:=&\sum_{s=2}^{\delta}{\delta\brace s}+\sum_{t=3}^{\delta}{\delta\brace t}q^{\delta-t},\\
F_2:=&\sum_{k=1}^{\delta-2}\binom{\delta}{k}\left(\sum_{s=1}^{\delta-k}{\delta-k\brace s}\right)\left(\sum_{t=1}^k{k\brace t}q^{k-t}\right),\\
F_3:=&\binom{\delta}{\delta-1}\sum_{t=2}^{\delta-1}{\delta-1\brace t}q^{\delta-1-t}.
\end{align*}
So we found the expected main term. We shall show that $F_1,F_2$ and $F_3$ can be absorbed into the error term.
\par\medskip\noindent
2) Taking into account that $3\leq\delta\leq\frac{\log q}{\log2}$ and ${\delta\brace t}\leq t^{\delta}(t!)^{-1}$, we obtain
\begin{equation*}
F_1\ll{\delta\brace 2}+q^{\delta-3}\left({\delta\brace 3}+\sum_{t=4}^{\delta}{\delta\brace t}q^{3-t}\right)\ll3^{\delta}q^{\delta-3},
\end{equation*}
and
\begin{equation*}
F_2\ll\sum_{k=1}^{\delta-2}\binom{\delta}{k}\left(\sum_{s=1}^{\delta-k}{\delta-k\brace s}\right)q^{k-1}\ll q^{\delta-3}\sum_{k=1}^{\delta-2}\binom{\delta}{k}\leq 2^{\delta}q^{\delta-3}
\end{equation*}
as well as
\begin{equation*}
F_3\ll2^{\delta}q^{\delta-3}\delta.
\end{equation*}
So we are done.
\end{proof}
                                                                                          
\begin{prop}
\label{3-5-3}
Let $q\in\mathbb{N}$ be sufficiently large and let $\Delta\in\mathbb{N}$ with $\Delta\leq \frac{\log q}{\log 2}$. In addition, let $n\geq \Delta q$ be an integer. Then
\begin{align*}
\frac{1}{n!}\sum_{\pi\in S_n}\binom{c_1(\pi^q)}{\Delta}=
\begin{cases}
\sigma_0(q)&\textit{if }\Delta=1\\[0.5ex]
\frac{1}{2}\sigma_1(q)+\mathcal{O}\left((\sigma_0(q))^2\right) &\textit{if }\Delta=2\\[0.5ex]
\frac{1}{6}\sigma_2(q)+\mathcal{O}(\sigma_0(q)\sigma_1(q)) &\textit{if }\Delta=3\\[0.5ex]
\frac{1}{\Delta!}\sigma_{\Delta-1}(q)+\mathcal{O}\left(\frac{1}{\Delta!}q^{\Delta-2}\left(\Delta\sigma_0(q)+2^{\Delta}\right)\right) &\textit{if }\Delta\geq 4.
\end{cases}
\end{align*}
\par\bigskip\noindent
The $\mathcal{O}$-constant is universal.
\end{prop}
\begin{proof}
1) It follows from Lemma \ref{3-3-4} that
\begin{equation*}
\Delta!\frac{1}{n!}\sum_{\pi\in S_n}\binom{c_1(\pi^q)}{\Delta}=\frac{1}{n!}\sum_{\pi\in S_n}(c_1(\pi^q))^{\Delta}+\mathcal{O}(F),
\end{equation*}
where
\begin{equation*}
F:=\sum_{\delta=1}^{\Delta-1}{\Delta\brack\delta}\frac{1}{n!}\sum_{\pi\in S_n}(c_1(\pi^q))^{\delta}.
\end{equation*}
The main term in the above formula is the mean we computed in Lemma \ref{3-5-1}. Applying this Lemma, we get the expected main and error term. Therefore, it only remains to be examined whether the error term $F$ is sufficiently small.
\par\medskip\noindent
2) Before we analyze the error term, we give two technical estimates. \\
\textit{Let  $\Delta\leq \frac{\log q}{\log 2}$. Then Lemma \ref{3-3-3} yields:}\\
\textit{i) For $\Delta\geq 3$ and $\delta\in\{1,2\}$ we have ${\Delta\brack\delta}\ll q^{\Delta-3}.$}\\
\textit{ii) For $\Delta\geq 4$ and $1\leq\delta\leq\Delta-1$ we get ${\Delta\brack\delta}\leq\Delta^2 q^{\Delta-\delta-1}.$ }
\par\medskip\noindent
3) Now we estimate the error term $F$. For $\Delta=1$ we obtain $F=0$. For $\Delta=2$ Lemma \ref{3-5-1} yields, that $F=\sigma_0(q)$. So let $\Delta\geq 3$. Applying Lemma \ref{3-5-1}, Lemma \ref{3-3-5} and the upper bounds of step 2) we get
\begin{equation*}
F\ll q^{\Delta-3}\sigma_1(q)+\sum_{\delta=3}^{\Delta-1}\Delta^2q^{\Delta-\delta-1}\sigma_{\delta-1}(q)\leq q^{\Delta-2}\left(\frac{\sigma_1(q)}{q}+\Delta^3\right),
\end{equation*}
which is sufficiently small.
\end{proof}
\begin{rem}
\label{3-5-4}
The error term in the preceding Proposition is essentially optimal: Confer Lemma \ref{3-5-2} and step 1) in the above proof.
\end{rem}

\section{The error term}
We implement the second step of our plan: We compute the error term in Proposition \ref{3-2-3}.

\begin{lem}
\label{3-6-1}
Let $q\in\mathbb{N}$ be sufficiently large and let $1\leq r\leq \exp(q^{1/3})$. In addition let $\delta\in\mathbb{N}$ with $\delta\leq \frac{\log q}{\log 2}$ and $n$ be a positive integer. Then
\begin{equation*}
\frac{1}{n!}\sum_{\pi\in S_n}\left(\sum_{d=1}^rc_d(\pi^q)\right)^{\delta}\ll
\begin{cases}
\sigma_0(q)H_r &\text{if }\delta=1\\[0.5ex]
(\sigma_0(q))^2H_r^2+\sigma_1(q)H_r &\text{if }\delta=2\\[0.5ex]
q^2H_r &\text{if }\delta=3\\[0.5ex]
q^{\delta-1}H_r^2 &\text{if }\delta\geq 4,
\end{cases}
\end{equation*}
where $H_r:=\sum_{d=1}^r\frac{1}{d}$. The $\mathcal{O}$-constant is universal.
\end{lem}
\begin{proof}
1) Similar to step 2) in the proof of Lemma \ref{3-5-1}, we obtain
\begin{equation*}
\frac{1}{n!}\sum_{\pi\in S_n}\left(\sum_{d=1}^rc_d(\pi^q)\right)^{\delta}\leq \sum_{l=1}^{\delta}\sum_{\{M_1,\ldots,M_l\}}\prod_{j=1}^l\left(\sum_{\substack{k\\k\leq (k,q)r}}(k,q)^{|M_j|}\sum_{s=1}^{|M_j|}{|M_j|\brace s}k^{-s}\right),
\end{equation*}
where the second sum on the right is over all set partitions of $\{1,\ldots\delta\}$ in exactly $l$ sets $M_1,\ldots,M_l$.\par\medskip\noindent
2) We draw our attention to
\begin{align*}
\sum_{\substack{k\\k\leq (k,q)r}}(k,q)^{m}\sum_{s=1}^{m}{m\brace s}k^{-s}&=\sum_{d=1}^r\sum_{s=1}^{m}\frac{1}{d^s}{m\brace s}\sum_{\substack{k\\k=(k,q)d}}\left(\frac{k}{d}\right)^{m-s}\\
&\leq\sum_{d=1}^r\frac{1}{d}\sum_{s=1}^m{m\brace s}\sigma_{m-s}(q)\\[0.6ex]
&=H_rT_m,
\end{align*}
where $T_m$ is defined as in step 3) in the proof of Lemma \ref{3-5-1}.
\par\medskip\noindent
3) The previous considerations in combination with Step 4) in the proof of Lemma \ref{3-5-1} yield
\begin{align*}
\frac{1}{n!}&\sum_{\pi\in S_n}\left(\sum_{d=1}^rc_d(\pi^q)\right)^{\delta}\\
&\leq H_r T_{\delta}+H_r^2\sum_{\{M_1,M_2\}}T_{|M_1|}T_{|M_2|}+\sum_{l=3}^{\delta}H_r^l{\delta\brace l}q^{\delta-l}\Bigl(\max\{3,\sigma_0(q)\}\Bigr)^l.
\end{align*}
Taking into account that $H_r\leq \log r+1$, our claim follows from step 3) and 7) in the proof of Lemma \ref{3-5-1} and from the estimates in Lemma \ref{3-3-1} and \ref{3-3-5}.
\end{proof}

\begin{lem}
\label{3-6-2}
Let $q\in\mathbb{N}$ be sufficiently large and let $\Delta\in\mathbb{N}$ with $\Delta\leq \frac{\log q}{\log 2}$. In addition, let $n$ be a positive integer. Then
\begin{equation*}
\frac{1}{n!}\sum_{j=1}^{\Delta}\sum_{\pi\in S_n}\binom{c_1(\pi^q)+\ldots+c_{j+1}(\pi^q)}{\Delta-j}\ll
\begin{cases}
1 &\text{if } \Delta=1\\[0.1ex]
\sigma_0(q) &\text{if } \Delta=2\\[0.1ex]
\sigma_1(q) &\text{if } \Delta=3\\[0.1ex]
\frac{1}{(\Delta-1)!}q^{\Delta-2} &\text{if }\Delta\geq 4.
\end{cases}
\end{equation*}
The $\mathcal{O}$-constant is universal.
\end{lem}

\begin{proof}
1) The case $\Delta=1$ is obvious. So let $\Delta\geq 2$. For $1\leq i\leq \Delta$ consider
\begin{align*}
Q(i,\Delta):&=\frac{i!}{n!}\sum_{\pi\in S_n}\binom{c_1(\pi^q)+\ldots+c_{\Delta-i+1}(\pi^q)}{i}\\
&=\sum_{\delta=1}^i (-1)^{i-\delta}{i\brack\delta}\frac{1}{n!}\sum_{\pi\in S_n}\left(\sum_{d=1}^{\Delta-i+1}c_d(\pi^q)\right)^{\delta}.
\end{align*}
The above equality is true due to Lemma \ref{3-3-4}. It follows with step 2) in the proof of Proposition \ref{3-5-3} and with Lemma \ref{3-6-1} that
\begin{equation*}
Q(i,\Delta)\ll
\begin{cases}
\sigma_0(q)H_{\Delta} &\textit{if }i=1\\
\sigma_1(q)H_{\Delta-1}^2&\textit{if }i=2\\
q^{i-1}H_{\Delta-i+1}^2&\textit{if }i\geq 3.
\end{cases}
\end{equation*}
In the case $i=2$ we also used the estimate $(\sigma_0(q))^2\leq 2\sigma_1(q)$.
\par\medskip\noindent
2) Finally, we look at
\begin{equation*}
R_{\Delta}:=\frac{1}{n!}\sum_{j=1}^{\Delta}\sum_{\pi\in S_n}\binom{c_1(\pi^q)+\ldots+c_{j+1}(\pi^q)}{\Delta-j}=1+\sum_{i=1}^{\Delta-1}\frac{Q(i,\Delta)}{i!}.
\end{equation*}
Due to the result of step  1), we  find that
\begin{equation*}
R_{\Delta}\ll
\begin{cases}
\sigma_0(q) &\textit{if }\Delta=2\\
\sigma_1(q) &\textit{if }\Delta=3\\
\frac{1}{(\Delta-1)!}q^{\Delta-2} &\textit{if }\Delta\geq 4.
\end{cases}
\end{equation*}
So we are done.
\end{proof}

\par\bigskip\noindent
\textbf{Author information}\\
Stefan-Christoph Virchow, Institut für Mathematik, Universität Rostock\\
Ulmenstr. 69 Haus 3, 18057 Rostock, Germany\\
E-mail: stefan.virchow@uni-rostock.de
\end{document}